\documentclass[11pt]{amsart}

\usepackage{amsfonts, amstext, amsmath, amsthm, amscd, amssymb}

\usepackage{graphicx, color}

\usepackage{microtype}

\usepackage[hidelinks,pagebackref,pdftex]{hyperref}

\newcommand{\bdy}{\partial}   

\newtheorem{theorem}{Theorem}[section]
\newtheorem{proposition}[theorem]{Proposition}
\newtheorem{lemma}[theorem]{Lemma}
\newtheorem{corollary}[theorem]{Corollary}

\newtheorem*{namedtheorem}{\theoremname}
\newcommand{\theoremname}{testing}

\theoremstyle{definition}
\newtheorem{definition}[theorem]{Definition}

\title{Hyperbolic Twisted Torus Links}
\author{Thiago de Paiva}
\address[]{School of Mathematics, Monash University, VIC 3800, Australia }
\email[]{thiago.depaivasouza@monash.edu}
 
\begin{document}

\begin{abstract}
The twisted torus link $T(p, q; r, s)$ is obtained by 
twisting $r$ parallel strands of the $(p, q)$-torus link a total of $s$ full times. In this paper we find all twisted torus links which are hyperbolic for $\vert s\vert >3$.
\end{abstract} 

\maketitle

\section{Introduction}
A \emph{torus link} is a link that can be embedded on the surface of an unknotted torus
$T$ in $S^3$. A torus link is denoted by $T(p,q)$ with $p, q$ integers. Here, $p$, $q$ denote the number of times that $T(p, q)$ wraps around the longitude, meridian, respectively, of $T$. The torus link $T(p, q)$  has $\gcd(p, q)$ link components. Thus if $\gcd(p, q) = 1$, then $T(p, q)$ is a knot.

The twisted torus knots $T(p, q; r, s)$ are obtained by performing a sequence of $s$ full twists on $r$ adjacent strands of $(p, q)$-torus knots. They were introduced by Dean in his doctoral thesis~\cite{Thesis} when he was investigating Seifert fibered spaces obtained by Dehn fillings.

Twisted torus knots have been found to have many interesting properties. They are known to be among those knots built with the least amount of tetrahedra \cite{simplest, nextsimplest}.
Their volumes ~\cite{generalizedtwistedtoruslinks}, knot Floer homology \cite{homology}, bridge spectra ~\cite{Bridge}, and Heegaard splittings~\cite{Heegaard} have been studied.

By work of Thurston~\cite{Thurston}, any knot in the 3-sphere can be classified geometrically as exactly one of a torus knot, a satellite knot, or a hyperbolic knot. 
The geometric classification of twisted torus knots has been studied by many authors; see particularly work of Lee~\cite{hyperbolicity, cable, LeeTorusknotsobtained, unknotted, Positively, torusTwistedtorusknots, Composite, Knottypes}, Morimoto~\cite{tangle}, Morimoto and Yamada~\cite{Anote}, Guntel~\cite{Guntel}, Thiago de Paiva~\cite{dePaiva:Unexpected, de2021hyperbolic, LeeThiago}, etc.

A natural generalization of twisted torus knots are twisted torus links, which are obtained by allowing $p, q$ to be not necessarily coprime in the definition of twisted torus knots. 

There has been great progress in the geometric classification of twisted torus knots, but there has been no result addressing the geometric classification of twisted torus links.

In this paper we classify all twisted torus links with $\vert s\vert > 3$ by proving the following theorem.

\begin{theorem}\label{maintheorem}
Consider $p, q, r$ integers greater than one such that $p+q \geq r >1$.
For $\vert s\vert > 3$, the twisted torus link $$T(p, q; r, s)$$ is hyperbolic 
if and only if $gcd(p, q) = 2$, $r$ is odd, different from $p\pm 1$ and $p$, and if $q>2$ then $r$ is not of the form $kq\pm 1$ with $k\geq 1$.
\end{theorem}

To prove this result, we consider the hyperbolization theorem proved by  Thurston \cite{Thurston}, which says: a compact 3-manifold $M$ with nonempty torus boundary
has interior admitting a complete hyperbolic structure if and only if it is
irreducible, boundary irreducible, atoroidal, and anannular. Irreducible, boundary irreducible, atoroidal, anannular means that $M$ doesn't contain any essential 2-sphere, disk, torus, annulus, respectively.
To prove that the twisted torus links with parameters not included in the hypotheses of Theorem~\ref{maintheorem}
are not hyperbolic, we find essential surfaces in their knot exterior. Then, we consider the hyperbolic twisted torus links of 
Theorem~\ref{maintheorem}.
To do this, we first study the hyperbolicity of the parental link associated with twisted torus link $T(p, q; r, s)$ which is the union of the torus link $T(p, q)$ and a circle $J$ encircling $r$ adjacent strands of this torus link. Then, we use some results of Gordon and Wu to study the slopes which produce hyperbolic Dehn fillings along $J$.

Positive twisted torus links, which are obtained by performing  positive full twists instead of negative full twists, are also T-links, as introduced by Birman and Kofman to describe Lorenz links ~\cite{newtwis}, for the particular case $r<p$. 
More precisely, if $s\geq 0$ and $r<p$, then the twisted torus link $T(p, q; r, s)$ is equal to the T-link $T((r, rs), (p, q))$.
In fact, as pointed out by Birman and Kofman \cite{newtwis}, T-links are a natural generalization of twisted torus links. 
Lorenz links are knotted closed periodic orbits in the flow of the Lorenz system. The Lorenz system is a system of three ordinary differential equations in $\mathbb{R}^3$ introduced by the meteorologist Edward Lorenz to predict weather patterns.
More specifically, Birman and Kofman proved that Lorenz links coincide with T-links. In the same paper they proved that twisted torus links are prime, fibered, nonamphicheiral, and have positive signature. They also establish some
dualities between positive twisted torus links and general twisted torus links and study the Jones polynomials of twisted torus links.

As an application of the last theorem, we obtain the following immediate corollary for the geometric classification of T-links. 

\begin{corollary}Consider $p, q, r$ integers greater than one such that $gcd(p, q)>1$ and $p > r >1$ . For $\vert s\vert > 3$, the T-link $$T((r, rs), (p, q))$$ is hyperbolic 
if and only if $gcd(p, q) = 2$, $r$ is odd and different from $p- 1$, and if $q>2$ then $r$ is not of the form $kq\pm 1$ with $k\geq 1$.
\end{corollary}

\subsection{Acknowledgment}The valuable discussions that I had with Jessica Purcell, my supervisor, and Sangyop Lee were very useful for this work.

\subsection{Data Availability} Data sharing is not applicable to this article as no datasets were generated or analysed during
the current study.

\section{Twisted Torus Links}

In this section we precisely define twisted torus links and set up notation. 

Throughout, let $F$ be an unknotted torus in $S^3$ containing the torus link $T(p,q)$, where $p$ and $
q$ are positive integers greater than one. Let $D$ be a disk intersecting $F$ in an arc with $r$ points of $T(p, q)$ in its interior, where $p + q \geq r \geq 2$. Let $J$ denote the boundary of $D$: $J=\bdy D$.

For a non-zero integer $s$, the twisted torus link $T(p, q; r, s)$ is obtained by performing $(1/s)$-Dehn surgery on the unknotted circle $J$.

\begin{definition}
  The manifold $S^3-(N(T(p, q)\cup J)$ is denoted by $M(p, q, r)$; here $N(\cdot)$ denotes a regular neighbourhood.
\end{definition}

If $r = p$, then the twisted torus link $T(p, q; r, s)$ is the torus link $T(p, q + rs)$. So, we assume that $r \neq p$.

When $p$ and $q$ are coprime, the twisted torus link $T(p, q; r, s)$ is the usual twisted torus knot $T(p,q;r,s)$ (its definition can also be seen in \cite{hyperbolicity}); note we are using the same notation.

If $gcd(p, q) = d$, then the torus link $T(p, q)$ has $d$ components. Thus, the twisted torus link $T(p, q; r, s)$ also has $d$ components.  

In this paper we will study the cases when $gcd(p, q) = d>1$.

The following proposition has \cite[Lemma 5]{Knottypes} as a specific case.

\begin{proposition}\label{Prop:(q,p)=(p,q)}
The twisted torus links $T(p, q; r, s)$ and $T(q, p; r, s)$ are equivalent.
\end{proposition}

\begin{proof}The manifold $M(p, q, r)$ is obtained from the link $T(p, q)\cup J$ with the circle $J$ encircling $r$ parallel strands of $T(p, q)$.
Consider that the torus link $T(p, q)$ lies in the unknotted torus $F$. After switching the meridian to the longitude and the longitude to the meridian of $F$, the torus link $T(p, q)$ becomes the torus link $T(q, p)$ and the circle $J$ encircles $r$ parallel strands of $T(q, p)$ as described in Figure~\ref{HTTL3}. Thus, after this isotopy the manifold $M(p, q, r)$ becomes the manifold $M(q, p, r)$. Therefore, $T(p, q; r, s)$ and $T(q, p; r, s)$ are also equivalent.
\end{proof}

\begin{figure}
\includegraphics[scale=0.38]{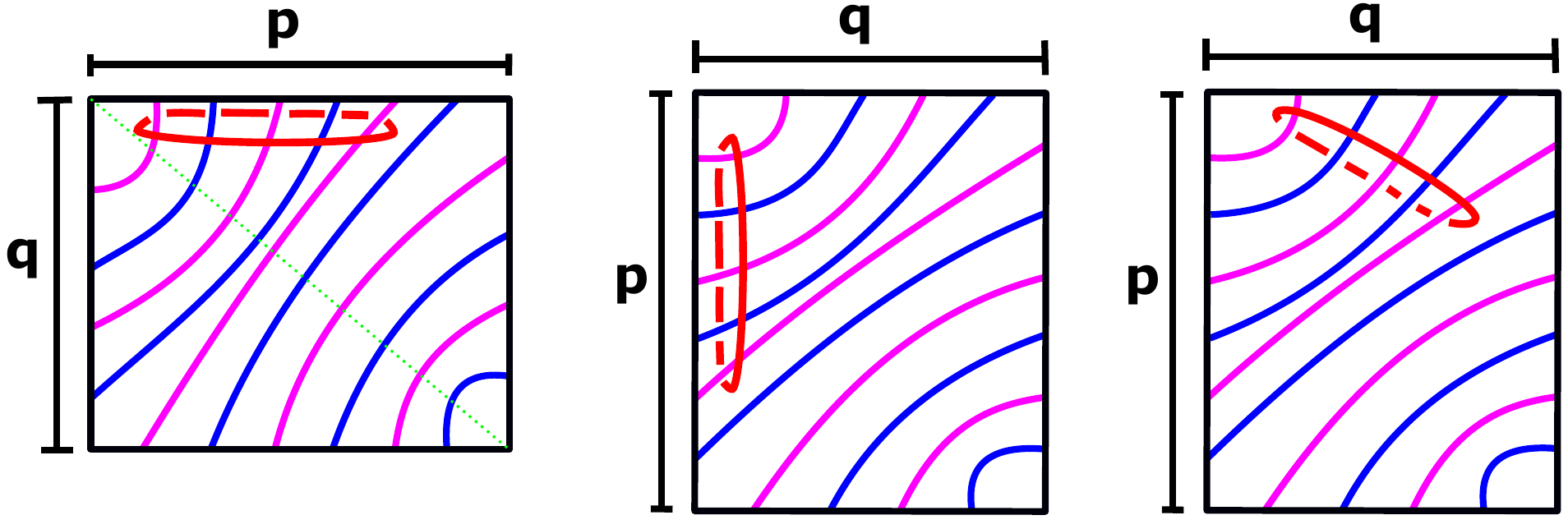} 
  \caption{The sides of the black square are glued together to give the unknotted torus $F$ in which the torus link $T(p, q)$ lies. The homeomorphism of $S^3$ that switches the meridian and longitude of $F$ can be expressed by a 180-degree rotation around the green diagonal of the first drawing, as illustrated in the second picture.}
  \label{HTTL3}
\end{figure}

Thus, from now on, we assume that $p\geq q$.
\section{Non-Hyperbolic Twisted Torus Links}
In this section we find the parameters $(p, q;r, s)$ which produce non-hyperbolic twisted torus links.

\begin{lemma}\label{linkingnumber}
Consider that $gcd(p, q) = d > 1$. Then, the linking number between any two link components of $T(p, q)$ is equal to $pq/d^2$.
\end{lemma}
\begin{proof}
Each link component of $T(p, q)$ is a $(p/d, q/d)$-torus knot whose exterior contains a winding annulus at slope $(p/d)(q/d)$. A second component of $T(p, q)$ lies in that winding annulus, hence it intersects a Seifert surface for the first component 
$(p/d)(q/d)$ times. Therefore, the linking number between those two components is $pq/d^2$.
\end{proof}

\begin{lemma}\label{irreducible,boundaryirreducible}
Consider that $gcd(p, q) = d>1$. Then, $S^3 - N(T(p, q))$ has no essential 2-spheres and no essential disks. Therefore, $S^3 - N(T(p, q))$  is irreducible and  boundary irreducible.
\end{lemma}

\begin{proof}By Lemma~\ref{linkingnumber}, any two link components of $T(p, q)$ have linking number greater than zero. Thus, $T(p, q)$ is nonsplittable, and so $S^3 - N(T(p, q))$ is irreducible.

If $S^3 - N(T(p, q))$ is boundary reducible, then there is a boundary reducible disk $D$ for $S^3 - N(T(p, q))$.  
Since $S^3 - N(T(p, q))$ is irreducible, it is not possible that $\partial D \subset \partial (S^3 - N(T(p, q)))$ does bound a disk $E$ on $\partial (S^3 - N(T(p, q)))$ but $E\cup D$ does not bound a 3-ball.
Consider $L_i$ the component of $\partial (S^3 - N(T(p, q)))$ in which $\partial D$ lies. Then, $\partial D$ does not bound a disk on $L_i$. Thus, $L_i$ would be trivial and $D$ would be its Seifert surface. Since $D$ is not punctured by the other components of $\partial (S^3 - N(T(p, q)))$, $L_i$ would have linking numbers equal to zero with the other components of $\partial (S^3 - N(T(p, q)))$, a contradiction.
\end{proof}

The manifold $F-N(T(p, q))$, where $F$ is the torus that the torus link $T(p, q)$ lies, is formed by $d$ annuli, where $gcd(p, q) = d>1$. We denote each of these annuli by $A^{i}_{T(p, q)}$ with $i = 1, \dots, d$. The annulus $A^{i}_{T(p, q)}$ has boundary components into two different link components of $T(p, q)$. 

\begin{lemma}\label{T(p, q)}
Each annulus $A^{i}_{T(p, q)}$ is essential in $S^3-N(T(p, q))$.
\end{lemma}

\begin{proof}Since $A^{i}_{T(p, q)}$  has boundaries in two different components, it is boundary incompressible and not boundary parallel. Also, the boundaries of $A^{i}_{T(p, q)}$ are essential in the tori that they lie. If there is a compression disk $D$ for $A^{i}_{T(p, q)}$, then we obtain two new disks, $D_1$ and $D_2$, by  doing a surgery in $A^{i}_T(p, q)$ along $\partial D$. We have that $\partial A^{i}_{T(p, q)} = \partial D_1 \cup \partial D_2$. Since $S^3-N(T(p, q))$ is irreducible and boundary irreducible by Lemma~\ref{irreducible,boundaryirreducible}, $D_1$ and $D_2$ bound disks on the tori that they lie, a contradiction. Thus, $A^{i}_{T(p, q)}$ is incompressible. Therefore, $A^{i}_{T(p, q)}$ is essential in $S^3-N(T(p, q))$.
\end{proof}

\begin{proposition}\label{greater2}
If $gcd(p, q) = d > 2$, then the twisted torus link $T(p, q; r, s)$ is not hyperbolic. 
\end{proposition}

\begin{proof}
There are $d$ annuli of type $A^{i}_{T(p, q)}$ in $S-N(T(p, q))$. 
From the definition of $J$, it intersects the torus $F$ only at two points, hence J intersects at most two of the annuli. Thus, there exists an annulus $A^{j}_{T(p, q)}$ disjoint from $J$. From Lemma~\ref{T(p, q)}, $A^{j}_{T(p, q)}$ is essential in $M(p,q,r)$.

Under $(1/s)$-Dehn filling along $J$, the annulus $A^{j}_{T(p, q)}$ is transformed into a new annulus, called $A$. The annulus $A$ has boundary components on distinct link components of $T(p,q;r,s)$, hence it is neither boundary parallel nor boundary compressible. The boundary of any compression disk is an essential loop in $A$,
hence parallel to one of the components of the link $T(p,q;r,s)$. Thus, if such a compression disk exists, the link is boundary compressible hence not hyperbolic.
On the other hand, if there is no compression disk, then the annulus $A$ is essential. It follows that the link $T(p,q;r,s)$ is not hyperbolic in this case as well.
\end{proof}

\begin{proposition}\label{equalto2}
If $gcd(p, q) = 2$ and $r$ is even, then the twisted torus link $T(p, q;r, s)$ is not hyperbolic.
\end{proposition}

\begin{proof}
As $gcd(p, q) = 2$, there are two annuli $A^{1}_{T(p, q)}$ and $A^{2}_{T(p, q)}$ associated with the torus link $T(p, q)$

Because $r$ is even, $J$ must be disjoint from one of these annuli, call it $A^{j}_{T(p, q)}$.  By Lemma~\ref{T(p, q)}, $A^{j}_{T(p, q)}$ is essential in $M(p,q,r)$.

As in the proof of Proposition~\ref{greater2}, the annulus $A^{j}_{T(p, q)}$ is transformed into a new annulus $A$
under Dehn surgery along $J$, with distinct boundary components. Hence it is neither boundary compressible nor boundary parallel. If it is compressible, then the compression disk has boundary parallel to a component of $T(p,q;r,s)$, so the link is boundary compressible and not hyperbolic. If $A$ is not compressible, then it is essential, and the link is not hyperbolic. 
\end{proof}

\begin{lemma}\label{LinkingNumber}
Suppose $\gcd(p,q)=2$, so that $T(p,q)$ has two components, denoted $L_1$ and $L_2$. Let $r>1$ be odd. Then the linking numbers $r_i = \mathrm{lk}(L_i,J)$ for $i=1,2$ satisfy:
\begin{itemize}
\item Both $r_1$ and $r_2$ are nonzero, with $r_1=r_2\pm 1$. Hence if $r_1$ is even, then $r_2$ is odd and vice versa.
\item The sum $r_1+r_2=r$.
\end{itemize}
\end{lemma}
\begin{proof}
The disk $D$ with $\bdy D=J$ is a Seifert surface for $J$ meeting $L_i$ a total of $r_i$ times, and $r$ is the total number of intersections of $D$ with $L_1\cup L_2$, hence $r=r_1+r_2$. Since $r$ is greater than one and odd, $D$ meets $L_1\cup L_2$ at least three times, alternating with intersections of $L_1$ and $L_2$ on $D$. Thus if the outermost intersection is with $L_1$, say, then $r_1=r_2+1$. Similarly if $L_2$ is outermost.
\end{proof}

\begin{proposition}\label{propositionp}
Consider $gcd(p, q) = 2$. If $r = p-1$ or $p+1$, then the twisted torus link $T(p, q;r, s)$ is not hyperbolic.
\end{proposition}

\begin{proof}
If $r = p-1$ or $p+1$, then one link component $L_i = T(\frac{p}{2}, \frac{q}{2})$ of the two-component link $T(p, q)$ has linking number equal to $\frac{p}{2}$ with $J$ from Lemma~\ref{LinkingNumber}.
It means that we can consider an unknotted torus $T$ bounding $L_i$ in one side and the other two link components of $M(p, q, r)$ in the other side so that $L_i$ can be embedded in $T$. Thus, $A = T-N(L_i)$ is an annulus. This annulus is essential in $M(p, q, r)$ because it is not boundary parallel in $N(L_i)$, and as $J$ wraps around $A$ one time, $A$ is also incompressible and boundary incompressible in $M(p, q, r)$.

After $(\frac{1}{s})$-Dehn surgery along $J$,  the link component $L_i$ is transformed into the torus knot $L'_i = T(\frac{p}{2}, \frac{q}{2} + s\frac{p}{2})$, but the torus $T$ remains the same. Thus, $A$ is transformed into the annulus $T - N(L'_i)$, called $A'$.

If $\frac{q}{2} + s\frac{p}{2} \neq -1$, then $L'_i$ is a non-trivial torus knot. Thus, $A'$ is essential in $S^3 - N(L'_i)$ \cite{Incompressible}, and so essential in $T(p, q;r, s)$ as well. This implies that $T(p, q;r, s)$ is not hyperbolic.

Consider now that $\frac{q}{2} + s\frac{p}{2} = -1$. We can assume that the other link component $L'_j$ of $T(p, q;r, s)$ has non-zero linking number with $L'_i$ otherwise $T(p, q;r, s)$ would be reducible, and so not hyperbolic. Thus, $L'_j$ wraps at least one time around $T$, which makes $A'$ incompressible and boundary incompressible. Since $A'$ is never  boundary parallel, $A'$ is essential in $T(p, q;r, s)$, implying that $T(p, q;r, s)$ is not hyperbolic in this case as well.
\end{proof}

\begin{figure}
\includegraphics[scale=0.65]{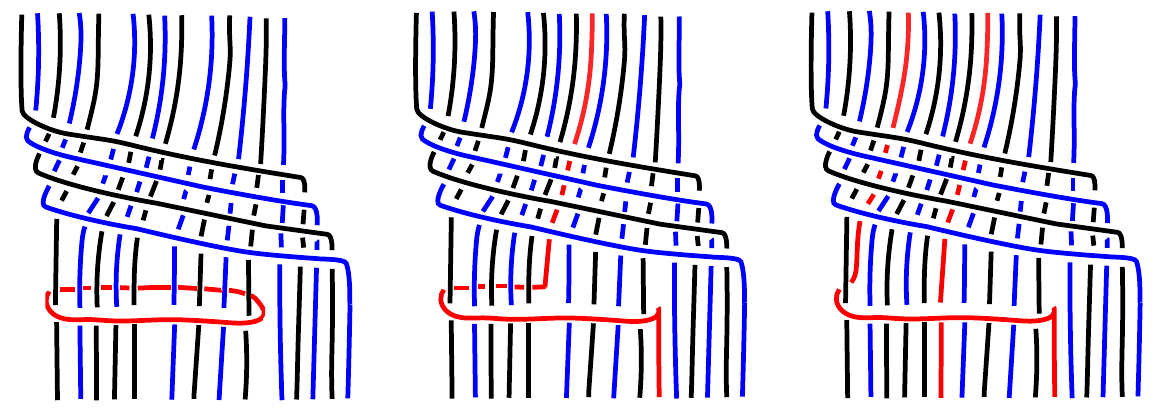} 
  \caption{If $r_i = k\cdot q/2$, we may isotope $J$ to embed in $S^3$ as the torus knot $T(k,1)$, disjoint from $L_1$ and $L_2$.}
  \label{Fig:MultQ1}
\end{figure}

\begin{proposition}\label{propositionq}
Consider $gcd(p, q) = 2$ with $q>2$ and $r$ odd. If $r_i$ is a multiple of $\frac{q}{2}$, then the twisted torus link $T(p, q;r, s)$ is not hyperbolic.
\end{proposition}

\begin{proof}Consider $T$ the torus which the torus link $T(p, q)$ lies.

Follow the procedure of Lee in \cite{cable} (also \cite{dePaivaPurcell:SatellitesLorenz}): drag the inner point of the intersection of $T$ and $J$ around the annulus $T-N(T(p,q))$. See Figure~\ref{Fig:MultQ1}. Because $r_i=kq/2$, after $k$ times, the curve connects to form a $(k,1)$-torus knot.

Let $F$ be an  unknotted torus containing $L_i$ in one side and the other two link components of $M(p, q, r)$ in the other side.  As $q>2$, $L_i$ wraps at least two times around $F$.
Thus, $F$ is essential in $M(p, q, r)$. Futhermore, $F$ is the essential torus obtained in \cite[Proposition 3.2]{dePaivaPurcell:SatellitesLorenz} for $S^3-N(L_i\cup J)$. Thus, by \cite[Theorem 4.3]{dePaivaPurcell:SatellitesLorenz} or   \cite[Theorem 2.1]{unknotted}, $(\frac{1}{s})$-surgery along $J$ transforms $F$ into a new torus $F'$ with its core being the torus knot $T(k, ks+1)$ and the component $L_i$ into the torus knot $T(\frac{q}{2}, \frac{p}{2} + k^2\frac{q}{2}s)$, called $L'_i$. As $\frac{q}{2}>1$, the wrapping number of $L'_i$ inside $F'$ is greater than one. So, $F'$ is not boundary parallel to $\partial N(L'_i)$.
Thus, if $k\neq 1$ or $(k, s)\neq (2, -1)$, then $F'$ is knotted and so essential in $S^3 - N(T(p, q;r, s))$. Consider now that $k = 1$ or $(k, s)= (2, -1)$. 
In this case, $F'$ is unknotted. If the other link component $L'_j$ of $T(p, q;r, s)$ is not isotopic to a meridian of $F'$, then $F'$ is essential in $S^3 - N(T(p, q;r, s))$. Consider now that $L'_j$ is isotopic to a meridian of $F'$. Then, the annulus $F'- N(L'_i)$ is essential in $S^3 - N(T(p, q;r, s))$. 

Therefore, in all cases, $T(p, q;r, s)$ is never hyperbolic.
\end{proof}

We summarize this section into the following theorem.

\begin{theorem}\label{main2}
The twisted torus link $T(p, q; r, s)$ is not hyperbolic for the following parameters:
\begin{enumerate}
\item $gcd(p, q) > 2$;

\item $gcd(p, q) = 2$ and $r$ even;

\item $gcd(p, q) = 2$ and $r = p-1$ or $p+1$;

\item $\gcd(p, q) = 2$ and $r$ equal to $kq\pm 1$ with $k\geq 1$ and $q>2$.
\end{enumerate}
\end{theorem}

\begin{proof}
It follows from Proposition~\ref{greater2}, \ref{equalto2}, \ref{propositionp}, and \ref{propositionq}.
\end{proof}

\section{Hyperbolic Twisted Torus Links}

In this section we prove that the twisted torus links $T(p, q; r, s)$ are hyperbolic for the remaining parameters $(p, q; r, s)$ with $\vert s\vert >3$.

From now on we always consider that $gcd(p, q) = 2$ and $r$ is odd, where $r$ is the linking number between $T(p, q)$ and $J$. The torus link $T(p, q)$ has two components, $L_1$ and $L_2$, both equal to the torus knot $T(\frac{p}{2}, \frac{q}{2})$.
Consider $r_i$ the linking number between $J$ and $L_i$.  If $r_1$ is odd, then $r_2$ is even, and vice versa by Lemma~\ref{LinkingNumber}. 
Furthermore, we assume that $r_i$ is not a multiple of $\frac{p}{2}$, and if $q>2$ then $r_i$ is also not a multiple of $\frac{q}{2}$.
If $r_i$ is a multiple of $\frac{p}{2}$, then $r = p\pm1$, which we have already proved to be non-hyperbolic in Proposition~\ref{propositionp}.
If $r_i$ is a multiple of $\frac{q}{2}$ with $q>2$, then $r = kq\pm 1$ for $k>0$, which is non-hyperbolic by Proposition~\ref{propositionq}.

$T(4,2;3,s)$ and $T(4,2;5,s)$ are the only twisted torus links for $p\leq 4$. We see that they don't satisfy the above criteria.
In fact, we proved that they are non-hyperbolic in Proposition~\ref{propositionp}. Thus, we also consider that $p>4$.
 
\begin{proposition}\label{irreducible}
$M(p, q, r)$ is irreducible and boundary irreducible
\end{proposition}

\begin{proof} 
Since the linking number of any two components of the link $T(p, q)\cup J$ is non-zero, the manifold $M(p, q, r)$ is irreducible. 

If $M(p, q, r)$ were boundary reducible, there would be a boundary reducible disk $D$ for $M(p, q, r)$. Consider $C$ the component of $\partial M(p, q, r)$ in which $\partial D$ lies. Then, $\partial D$ does not bound a disk on $C$ . Thus, $C$ would be trivial and $D$ would be its Seifert surface. Since $D$ is not punctured by the other components of $\partial M(p, q, r)$, $C$ would have linking numbers equal to zero with the other components of $\partial M(p, q, r)$, a contradiction.
\end{proof}

We will prove that $M(p, q, r)$ doesn't have any essential annuli in the next two subsections. We will break this proof into two cases. But, we now consider that the following proposition is true.

\begin{proposition}\label{annuli133}
Consider $p, q, r$ integers greater than one such that $p+q \geq r >1$ and $p\neq r$.
Furthermore, assume that $gcd(p, q) = 2$, $r$ is odd and different from $p\pm 1$, and if $q>2$ 
then $r$ is not of the form $kq\pm 1$ with $k\geq 1$. Then, $M(p, q, r)$ is anannular.
\end{proposition}  
\begin{proof}   
It follows from Proposition~\ref{noessentialannuli} and \ref{anannular2}. 
\end{proof}  

\begin{lemma}\label{intersects}
The circle $J$ intersects each $A^{i}_{T(p, q)}$ at least one time.
\end{lemma} 

\begin{proof}If $J$ could be isotopic to miss $A^{i}_{T(p, q)}$, then the annulus $A^{i}_{T(p, q)}$ would be an essential annulus in $M(p, q, r)$ by Lemma~\ref{T(p, q)}, contradicting the last proposition.
\end{proof}

Denote by $\overline{A}^{i}_{T(p, q)}$ the surface $A^{i}_{T(p, q)}\cap M(p, q, r)$.
Since $J$ intersects each $A^{1}_{T(p, q)}$ and $A^{2}_{T(p, q)}$ at least once and the torus $F$ twice, it follows that $J$ intersects each $A^{1}_{T(p, q)}$ and $A^{2}_{T(p, q)}$ exactly once. Thus, $A^{i}_{T(p, q)}\cap M(p, q, r)$ is a once punctured annulus (or a twice punctured disk).
 
\begin{lemma} 
Each surface $\overline{A}^{i}_{T(p, q)}$ is essential in $M(p, q, r)$.
\end{lemma}

\begin{proof}Consider $D$ a compression disk for $\overline{A}^{i}_{T(p, q)}$. So, $\partial D$ is essential in $\overline{A}^{i}_{T(p, q)}$. Thus, $\partial D$ is always isotopic to one boundary component of $\overline{A}^{i}_{T(p, q)}$. This implies that $D$ is a boundary reducible disk for $M(p, q, r)$, which is a contradiction with Proposition~\ref{irreducible}. Thus, $\overline{A}^{i}_{T(p, q)}$ is incompressible in $M(p, q, r)$. 
Therefore, since an incompressible surface inside an irreducible 3-manifold,
with boundary contained in the torus boundary components of the 3-manifold, is essential
unless the surface is a boundary parallel annulus,  $\overline{A}^{i}_{T(p, q)}$ is essential in $M(p, q, r)$.
\end{proof}

\begin{proposition}\label{atoroidal22}
Consider $p, q, r$ integers greater than one such that $p+q \geq r >1$ and $p\neq r$.
Furthermore, assume that $gcd(p, q) = 2$, $r$ is odd and different from $p\pm 1$, and if $q>2$ 
then $r$ is not of the form $kq\pm 1$ with $k\geq 1$. Then, $M(p, q, r)$ is atoroidal.
\end{proposition}
\begin{proof}Consider $T$ an essential torus in $M(p, q, r)$.

If $T\cap (\overline{A}^{1}_{T(p, q)}\cup \overline{A}^{2}_{T(p, q)}) = \varnothing$, then $T$ is trivial containing $L_1$ and $L_2$ on the same side.
By Lemma~\ref{intersects}, $J$ must be on the same side as $L_1$ and $L_2$, which makes $T$ compressible, a contradiction. Thus, $T$ intersect at least one $\overline{A}^{i}_{T(p, q)}$. We also consider that $T\cap \overline{A}^{i}_{T(p, q)}$ is transverse. Furthermore, $T\cap \overline{A}^{i}_{T(p, q)}$ consists of essential loops in $\overline{A}^{i}_{T(p, q)}$ as $T$ is incompressible. 

Consider $\gamma$ an innermost loop of $T \cap \overline{A}^{i}_{T(p, q)}$ in $\overline{A}^{i}_{T(p, q)}$. First assume that $\gamma$ is isotopic to one boundary component $C$ of $\overline{A}^{i}_{T(p, q)}$. Then, $\gamma$ and $C$ bounds an annulus $A$ in $\overline{A}^{i}_{T(p, q)}$ which intersects $T$ just in $\gamma$. By doing a surgery in $T$ along $A$, we obtain a new properly embedded annulus $A'$ which has boundaries in $C$. $A'$ is incompressible and not boundary parallel in $M(p, q, r)$ since $\overline{A}^{i}_{T(p, q)}$ is incompressible and $T$ is not boundary parallel. So, $A'$ is essential in $M(p, q, r)$, which is a contradiction with Proposition~\ref{annuli133}. Finally, if $\gamma$ cuts a pair of pants containing two boundary components  of $\overline{A}^{i}_{T(p, q)}$, then the outermost loop of $T\cap \overline{A}^{i}_{T(p, q)}$ is isotopic to the third boundary component of $\overline{A}^{i}_{T(p, q)}$, which is not possible by the previous argument.
\end{proof} 

\begin{theorem}
Consider $p, q, r$ integers greater than one such that $p+q \geq r >1$ and $p\neq r$.
Furthermore, assume that $gcd(p, q) = 2$, $r$ is odd and different from $p\pm 1$, and if $q>2$ 
then $r$ is not of the form $kq\pm 1$ with $k\geq 1$. Then, $M(p, q, r)$ is hyperbolic.
\end{theorem}

\begin{proof}
It follows from Proposition~\ref{irreducible}, \ref{annuli133} and \ref{atoroidal22}. 
\end{proof}

\begin{theorem}\label{main1}
Consider $p, q, r$ integers greater than one such that $p+q \geq r >1$ and $p\neq r$.
Furthermore, assume that $gcd(p, q) = 2$, $r$ is odd and different from $p\pm 1$, and if $q>2$ 
then $r$ is not of the form $kq\pm 1$ with $k\geq 1$. Then, for $\vert s\vert\geq 4$, the twisted torus link $$T(p, q; r, s)$$ is hyperbolic.
\end{theorem}

\begin{proof}
Consider that the twisted torus link $T(p, q; r, s)$ has an essential torus for $\vert s\vert \geq 4$. For $s=0$, the twisted torus link $T(p, q; r, s)$ has an essential annulus by Lemma~\ref{T(p, q)}. The manifold $M(p, q, r)$ is hyperbolic from the last theorem. Then, by [Theorem 1.1, Gordon and Wu \cite{GordonToroidal}], $\vert s\vert < 3$ (since $\Delta(1/0, 1/s) = \vert s \vert$ and $M(p, q, r)$ is not homeomorphic to $M_1$, $M_2$, or $M_3$ in \cite[Theorem 1.1]{GordonToroidal} as these manifolds are hyperbolic), a contradiction.

Consider now that the twisted torus link $T(p, q; r, s)$ has an essential annulus for $\vert s\vert \geq 4$. Then, $\vert s\vert < 3$ by [Theorem 1.1, Gordon and Wu \cite{GordonAnnular}] (because the exteriors of the links in the items (1), (2), and (3) of \cite[Theorem 1.1]{GordonAnnular} are $M_1$, $M_2$, and $M_3$, respectively, in \cite[Theorem 1.1]{GordonToroidal}). 

Lastly, $T(p, q; r, s)$ is irreducible by [Theorem 5.1, Wu \cite{Wulaminations}] and boundary irreducible by [Theorem 1, Gordon and Wu \cite{GordonWu}].
\end{proof}

Therefore, Theorem~\ref{maintheorem} follows from Theorem ~\ref{main2} and \ref{main1}.

\subsection{First Case: $q > 2$}  
In this subsection we prove Proposition~\ref{annuli133} for when $q > 2$. Therefore, we also assume $q > 2$ in this subsection.

The two link components of $T(p,q)$, $L_1$ and $L_2$, are both the nontrivial torus knot $T(p/2,q/2)$.
Denote $S^3-(N(L_i\cup J))$ by $M_i$ with $i= 1, 2$. Lee proved that $M_i$  is hyperbolic if $r_i>1$ in \cite[Proposition 5.7]{unknotted}.

\begin{lemma}\label{annulitwocomponent}
$M(p, q, r)$ contains no essential annuli with boundary components in two different link components.
\end{lemma}

\begin{proof}
Consider $A$ an annulus which has boundaries in $L_1$ and $L_2$. Assume that $A$ intersects $A^{i}_{T(p, q)}$  with $i = 1$ or $2$. Then, since $A$ and $A^{i}_{T(p, q)}$ are incompressible, $A\cap A^{i}_{T(p, q)}$ are formed by essential loops in both $A$ and $A^{i}_{T(p, q)}$. Thus, each loop in $A\cap A^{i}_{T(p, q)}$ is isotopic to the boundary components of $A$ and $A^{i}_{T(p, q)}$. Therefore, the boundaries of $A$ are isotopic to the boundaries of  $A^{i}_{T(p, q)}$. Since the boundaries of  $A^{i}_{T(p, q)}$ are isotopic to $L_1$ and $L_2$, the boundaries of $A$ are also isotopic to $L_1$ and $L_2$. Thus,  $L_1$ and $L_2$ are isotopic to each other in $M(p, q, r)$ throughout $A$, which implies $r_1 = r_2$ contradicting Lemma~\ref{LinkingNumber}. Now if $A$ doesn't intersect $A^{i}_{T(p, q)}$ with $i = 1, 2$, then if we push the boundaries of $A$ to the interiors of $N(L_1)$ and $N(L_2)$, they must have wrapping numbers equal to one in these solid tori. This implies that the boundaries of $A$ are isotopic to $L_1$ and $L_2$, which generates a contradiction as before. 

By \cite[Lemma 5.2]{unknotted}, there is no annuli with boundaries in $\partial_{L_i}M$ and $\partial_{J}M$ if $r_i>1$. Then, consider $A$ an annulus in $M(p, q, r)$ with boundaries in  $\partial_{L_i}M$ and $\partial_{J}M$
with $r_i=1$. So, $r_j=2$ and $r=3$. 
The boundary $\partial_1 A$ of $A$ in $\partial_{J}M$ should be a torus knot $T(a, b)$. If 
$T(a, b)$ is trivial, then $\partial_2 A$ is a meridian of $\partial_{L_i}M$ as $L_i$ is a non-trivial torus knot. As $lk(\partial_2 A, J) = 0$, then $\partial_1 A$ is the longitude of $\partial_{J}M$ because $\partial_1 A$ and $\partial_2 A$ are isotopic through $A$. So, $\partial_1 A$ and $\partial_{J}M$ are isotopic. 
This implies that $lk(L_j, J) = 0$, a contradiction. 
Hence $T(a, b)$ is non-trivial. Because torus knots don't have any essential torus \cite{Incompressible}, $\partial_2 A$ is isotopic to $L_i$. Then, $L_i$ can be embedded in $\partial_{J}M$. Since $lk(L_i, J) = 1$, then $b= \pm 1$  \cite{Rolfsen}. However, this implies that $L_i$ is trivial, a contradiction.
\end{proof}

Observe that up to isotopy, there is exactly one essential annulus in $S^3-N(L_i)$ with both of its boundary components on $\bdy N(L_i)$, for $i=1,2$: this is the annulus $F-N(L_i)$, where $F$ is the unknotted torus that the torus knot $L_i$ lies. See~\cite{Incompressible}.

In $S^3-N(L_1\cup L_2)$, there may be more essential annuli with both boundary components on $\bdy N(L_i)$ and essential in $S^3-N(L_i)$. For example, isotope the link in a regular neighbourhood $T\times[0,1]$ of $T$ so that $L_1$ lies on $T\times\{0\}$ and $L_2$ on $T\times\{1\}$. The annulus $P_1' = T\times\{0\}-L_1$ is one such annulus. Alternatively, we might isotope $L_1$ to $T\times\{1\}$ and $L_2$ to $T\times\{0\}$ and let $P_1''=T\times\{1\}-L_1$. These are not obviously isotopic, because $L_2$ lies between them. In any case, in the lemma below, we may take $P_i'$ to be any annulus in $S^3-N(L_1\cup L_2)$ with both boundary components on $\bdy N(L_i)$ that is essential in $S^3-N(L_i)$. Observe that in the above example, $J$ meets $P_1'$ and $P_1''$ exactly twice. The point of the next lemma is to show that $J$ must meet any such annulus at least two times.

\begin{lemma}\label{minimally12}
Let $P_i'$ be an annulus in $S^3-N(L_1\cup L_2)$ as above: namely, $P_i'$ is essential in $S^3-N(L_i)$.  
Then for $i=1,2$, $P_i'$ must intersect $J$ at least twice.
\end{lemma}

\begin{proof}
We assume $P'_i$ intersects $J$ transversely.

Note that $\bdy P_i'$ cuts $\bdy N(L_i)$ into two annuli. Call one such annulus $B$. Then $P_i'\cup B$ is an unknotted torus in $S^3$, hence separating. Since $J$ does not meet $B\subset \bdy N(L_i)$, $J$ must meet $P_i'$ an even number of times. 

Assume $\vert P_i'\cap J\vert = 0$. The torus $P_i'\cup B$ splits $S^3$ into two solid tori. Then $J$ can be isopoted into one of these solid tori, call it $V$.

Consider the winding number $\omega$ of $J$ in $V$. Observe that the linking number $r_i$ must be equal to $\omega\frac{p}{2}$ or $\omega\frac{q}{2}$ depending on which side of $T$ the solid torus $V$ lies. This contradicts our assumption on linking numbers since $r_i$ is not a multiple of $\frac{p}{2}$ and $\frac{q}{2}$.

Therefore, $\vert P_i'\cap J\vert \geq 2$.
\end{proof}

Observe that when we put $L_1$ on $T\times\{0\}$ and $L_2$ on $T\times\{1\}$, then there is a torus $T\times\{1/2\}$ between the two link components of $T(p,q)$ that is disjoint from $L_1$, $L_2$, and divides $S^3$ into two solid tori, each containing one torus knot $L_i$, hence $T\times\{1/2\}$ is essential in $S^3-N(T(p,q))$. Observe also that $T\times\{1/2\}$ is unknotted and doesn't intersect the essential annuli $P'_1 = T\times\{0\} - L_1$ and $P'_2 = T\times\{1\}- L_2$. The next lemma considers any unknotted essential torus in $S^3-N(T(p,q))$. 

\begin{lemma}\label{onthesameside} 
For any unkotted essential torus $T$ in $S^3-N(T(p,q))$, there is an annulus $P'_i$ in $S^3-N(T(p,q))$ which has boundary on $\bdy N(L_i)$, is essential in $S^3-N(L_i)$, and is disjoint from $T$ for each $i=1,2$. 
\end{lemma}

\begin{proof}
Let $T$ be an unknotted essential torus in $S^3-N(T(p, q))$. Fix an annulus $A_i$ as in Lemma~\ref{minimally12}. If $T\cap A_i = \emptyset$, then we are done. Otherwise,  $T$ intersects $A_i$ in simple closed curves. Because $T$ is incompressible, we may isotope so that no loop of $T\cap A_i$ bounds a disk in $A_i$. So all loops of $T\cap A_i$ are isotopic to the core of the annulus $A_i$, implying they are isotopic to $L_i$. Use this isotopy to embed $L_i$ in $T$. Then isotope $T-L_i$ slightly to give an annulus $P_i'$ that is disjoint from $T$.
\end{proof}

\begin{lemma}\label{justone}
If $M(p, q, r)$ has an essential annulus $A$, then $\partial A$ is contained in just one boundary component $B$ of  $M(p, q, r)$. Furthermore, $A$ is inessential in the exterior of $B$.
\end{lemma}

\begin{proof} 
Suppose that $M(p, q, r)$ has an essential annulus $A$.
By Lemma~\ref{annulitwocomponent}, $\partial A$ is contained in just one of $\partial_{L_1}M$, $\partial_{L_2} M$, or $\partial_{J} M$. Call $B$ this boundary component in which $\partial A$ lies. 

If $B = \partial_{J} M$, then $A$ is not essential in the exterior of $B$ because the unknot torus has no essential annuli \cite[page 15]{hatcher2007notes}.
 
Consider now that $B$ is equal to $\partial_{L_i} M$ with $i = 1$ or $2$. If $A$ is essential in the exterior of $B$, then $A$ is equal to $F-B$ where $F$ is the unknotted torus in which the torus knot $B$ lies by \cite{Incompressible}. 
However, by Lemma~\ref{minimally12}, $J$ must intersect $A$, a contradiction.
\end{proof}

\begin{proposition}\label{noessentialannuli}
Consider $p, q, r$ integers greater than two such that $p+q \geq r >1$. Furthermore, assume that $gcd(p, q) = 2$ and $r$ is odd, different from $p\pm 1$, $p$, and $kq\pm 1$ for $k\geq 1$. Then, $M(p, q, r)$ is anannular.
\end{proposition}
\begin{proof}Suppose $M(p, q, r)$ contains an essential annulus $A$.
By the last lemma, $\partial A$ is contained in just one boundary component $B$ of  $M(p, q, r)$ and $A$ is inessential in the exterior of $B$. We denote by $M_B$ the exterior of $B$ in $S^3$.
  
Since the annulus $A$ is inessential in $M_B$, then $A$ is compressible, boundary compressible, or boundary parallel in $M_B$.

If $A$ is boundary parallel to an annulus $A'$ in $M_B$, then $A\cup A'$ bounds a solid torus $V$ in $M_B$. Since  $A$ is not boundary parallel in $M(p, q, r)$, one component $C$ of $\partial M(p, q, r)$ must be inside $V$. If there is a meridian disk $D$ of $V$ which does not intersect $C$, then the surgery of $\partial V = A\cup A'$  along $D$ yields a sphere containing $C$ in one side and the boundary component which contains $\partial A$  in the other side. Thus, $M(p, q, r)$ would be reducible, contradicting Lemma~\ref{irreducible}. So, the wrapping number of $C$ in $V$ is greater than zero. 

Consider first $A'\subset N(J)$. Then, $C$ is equal to one link component $L_j$ of $T(p, q)$.
If $\partial V$ is boundary parallel to $L_j$, then the core of $V$, which is the same as the core of $A'$, is $L_j$.
So, $L_j$ can be embedded in the unknotted torus $\partial N(J)$. Thus, $L_j$ can be described as a $(\pm \frac{p}{2}, \pm \frac{q}{2})$ or $(\pm \frac{q}{2}, \pm \frac{p}{2})$-torus knot on $\bdy N(J)$ because equivalent torus knots have been
classified; see \cite{Rolfsen}. This implies that the linking number $r_j$ of $L_j$ and $J$ is $\frac{p}{2}$ or $\frac{q}{2}$, which is a contradiction.

If $\bdy V$ is not an unknotted torus in $S^3$, then it forms an essential torus in $S^3-N(L_j)$. But there are no essential tori in torus knot exteriors by \cite{Incompressible}. 
Thus, $\partial V$ is a trivial torus in $M(p, q, r)$. 

Consider that the other link component $L_i$ of $T(p, q)$ is also contained in $V$. Then, the core of $A'$ is not the trivial torus knot $T(a, \pm 1)$ otherwise the disk with boundary the longitude of $\partial N(J)$ is a boundary compressible disk for $A$ in $M(p, q, r)$, which is impossible. Thus, the core of $A'$ is the trivial torus knot $T(\pm 1, b)$ with $\vert b \vert \neq 1$.
If $b = 0$, then $L_i$ and $L_j$ would have zero linking numbers with $J$, a contradiction. Thus, $\vert b \vert > 1$.
Then, $\partial V$ is essential in $M_j$ and $M_i$, but this is not possible since at least one of these manifolds is hyperbolic by \cite[Proposition 5.7]{unknotted}. So, $L_i$ is not contained in $V$. Thus, $\partial V$  is essential in $S^3 - N(T(p, q))$.
By Lemma~\ref{onthesameside}, there exists an annulus $P'_j$, essential in $S^3-N(L_j)$, such that $\bdy V$ does not intersect $P'_j$. Observe that $P'_j$ must therefore lie entirely inside of $V$. So, $P'_j$ does not intersect $J$, contradicting Lemma~\ref{minimally12}.

Now assume that $A'\subset N(L_j)$. Consider first that $J$ is inside $V$. Then, $A'$ is a meridional annulus of $N(L_j)$ otherwise $\partial V$ would be an essential torus for $J$, which is impossible. 
If the other link component $L_i$ of $T(p, q)$ is outside $V$, then the linking number between $L_i$ and $J$ is zero, a contradiction. So, $L_i$ is inside $V$. Hence, $\partial V$ is an unknotted essential torus in $S^3 - N(T(p, q))$. 
Lemma~\ref{onthesameside} says that there exists an annulus $P'_j$, essential in $S^3-N(L_j)$, such that $\bdy V$ does not intersect it. Then, $P'_j$ must lie in the outside of $V$. So,  $J$ does not intersect $P'_j$, which is impossible by Lemma~\ref{minimally12}.
Thus, $J$ doesn't live in $V$ and only $L_i$ lies in $V$.
If $\partial V$ is boundary parallel to $N(L_i)$, then $L_i$ is the core of $V$. Thus, we can embed $L_i$ in $\partial N(L_j)$. Then, $L_i$ is isotopic to $L_j$ in $M(p, q, r)$ since a torus knot is not a satellite knot \cite{Incompressible}, which implies $r_1 = r_2$ contradicting Lemma~\ref{LinkingNumber}. Thus, $\partial V$ is trivial, which implies that $A'$ is a meridional annulus of $N(L_j)$, and so the linking number between $L_i$ and $J$ is zero, a contradiction. Therefore, $A$ can't be boundary parallel in $M_B$. 

Assume now that $A$ is compressible in $M_B$. Then, there is a compression disk $D$ for $A$ in $M_B$. The surgery of $A$ along $D$ yields two disks, $D_1$ and $D_2$, such that $\partial A = \partial D_1 \cup \partial D_2$.
Assume first that $B$ is equal to one $\partial_{L_i} M(p, q, r)$. 
Since the exterior of $L_i$ is boundary irreducible, each $\partial D_i$ bounds a disk in $\partial N(L_i)$. Thus, by considering a disk with boundary in $A$ close to $\partial D_i$, we can see that $A$ is also compressible in $M(p, q, r)$, a contradiction.
Consider now that $B$ is equal to $\partial_J M(p, q, r)$.
If one $\partial D_i$ bounds a disk on $\partial N(J)$, then we obtain a contradiction as before. 
Now suppose that none $\partial D_i$ bounds a disk on $\partial N(J)$. Then, both $\partial D_1$ and $\partial D_2$ are either isotopic to the longitude or meridian of $\partial N(J)$ implying that $A$ is also boundary parallel in $M_B$. However, we have already ruled out this possibility.

Therefore, as $A$ is an incompressible and not boundary parallel surface in $M_B$, $A$ is essential in $M_B$, which is not possible by Lemma~\ref{justone}.
\end{proof} 

\subsection{Second Case: $q = 2$}
Finally we prove Proposition~\ref{annuli133} for when $q = 2$. Thus, we consider that $q = 2$ in this subsection.

The two link components of $T(p,2)$, $L_1$ and $L_2$, are both equal to the trivial torus knot $T(\frac{p}{2}, 1)$.
Denote $S^3-N(L_i\cup J)$ by $M_i$ with $i= 1, 2$.

\begin{proposition}\label{noessentialtorus}
Consider $a>1$ and $a\geq b>0$. The exterior of the link $T(a, 1)\cup C$, where $T(a, 1)$ is the $(a, 1)$-trivial torus knot and $C$ is a circle encircling $b$ longitudinal strands of $T(a, 1)$, as shown on the left of Figure~\ref{1}, has no essential tori.
\end{proposition}

\begin{proof}Consider first $b>1$. By removing the last $a-b$ crossings of $T(a, 1)$, we can transform the knot $T(a, 1)$ into the knot $T(b, 1)$ and $C$ into the braid axis of the braid of the last knot. 

If the exterior of $T(b, 1)\cup C$ has an essential torus, then $T$  doesn't  intersect $C$ and bounds a solid torus containing $T(b, 1)$ in one side and $C$ in the other. Futhermore, since the unknot doesn't have any essential tori \cite[page 15]{hatcher2007notes}, $T$ is unknotted. 

After $(1/k)$-surgery along $C$, we obtain the $(b, bk+1)$-torus knots and a new torus $T'$ in the exterior of this torus knot. If we consider high slopes, we can assume that there is $(1/s)$-surgery along $C$ such that $T'$ is knotted. As $T$ is incompressible and not boundary parallell in the side containing $T(b, 1)$, $T'$ is also incompressible and not boundary parallell in the side containing the non-trivial torus knot $T(b, bs+1)$. Thus, $T'$ is an essential torus for $S^3 - N(T(b, bs+1))$. But, it implies a contradiction since torus knots don't have any essential tori by \cite{Incompressible}.

Now if $b=1$, then $T(a, 1)\cup C$ is the Hopf link which doesn't have any essential tori.
\end{proof}

\begin{corollary}\label{atoroidaltoruslink2} 
$M_1$ and $M_2$ are atoroidal.
\end{corollary}

\begin{proof}
It follows from the last proposition since $M_i = S^3-N(T(\frac{p}{2}, 1)\cup J)$, where $J$ is a circle encircling $r_i$ longitudinal strands of $T(\frac{p}{2}, 1)$ with $r_i< \frac{p}{2}$ (as $r = r_1 + r_2$ is odd, different from $p\pm 1$, and less than or equal to $ p + 2$).
\end{proof}

\begin{lemma}\label{atoroidaltoruslink}
The torus link $T(p, 2)$ is atoroidal.
\end{lemma}

\begin{proof}
Consider the link $T(\frac{p}{2}, 1)\cup C$, where $C$ is the braid axis of $T(\frac{p}{2}, 1)$. By pushing the circle $C$ following the longitudinal strands of $T(\frac{p}{2}, 1)$, as illustrated in Figure~\ref{1}, we can see that the links $T(\frac{p}{2}, 1)\cup C$ and $T(p, 2)$ are equivalent. Now, the result follows from Proposition~\ref{noessentialtorus}.
\end{proof}

\begin{figure}
\includegraphics[scale=1]{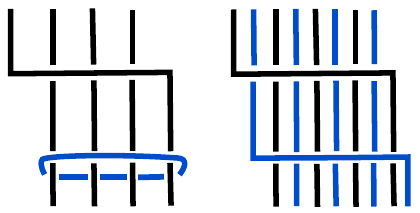} 
  \caption{The blue disk is encircling 4 strands of the trivial torus knot $T(4, 1)$ in the first drawing. The second drawing is obtained by pushing the blue disk following the longitudinal strands of $T(4, 1)$.}
  \label{1}
\end{figure}

\begin{lemma}\label{annulitwocomponent2}
$M(p, 2, r)$ contains no annuli with boundary components in two different link components.
\end{lemma}

\begin{proof}
Consider $A$ an annulus with boundary components in two different link components. 

First assume that these link components are $\partial_{L_1}M(p, 2, r)$ and $\partial_{L_2}M(p, 2, r)$. Since $\partial_{L_1}M(p, 2, r)$ and $\partial_{L_2}M(p, 2, r)$ are trivial tori, the boundary components 
$\partial A_1 \subset \partial_{L_1}M(p, 2, r)$ and $\partial A_2 \subset \partial_{L_2}M(p, 2, r)$ of $A$ are torus knots.

Consider first that $\partial A_1$ is the non-trivial torus knot $T(a, b)$. 
Then, $\partial A_2$ is any torus knot isotopic to $T(a, b)$  since $\partial A_1$ and $\partial A_2$ are isotopic to each other through $A$. Thus, $\partial A_2$ is equal to $T(a, b)$, $T(-a, b)$, $T(a, -b)$, $T(-a, -b)$, $T(b, a)$, $T(-b, a)$, $T(b, -a)$, or $T(-b, -a)$ \cite{Rolfsen}.
The linking number of $\partial A_1$ and $L_1$ is $b$. 
Because the linking number of $L_1$ and $L_2$ is  $\frac{p}{2}$ by Lemma~\ref{linkingnumber}, the linking number of $L_1$ and $\partial A_2$ is $a\frac{p}{2}$, $b\frac{p}{2}$, $-a\frac{p}{2}$, or $-b\frac{p}{2}$.
If  $b =  b\frac{p}{2}$ or $-b\frac{p}{2}$, then $\frac{p}{2} = \pm 1$, which is not possible since $p\geq 5$. 
If  $b =  a\frac{p}{2}$ or $-a\frac{p}{2}$, then $a, b$ would not be coprime, a contradiction.

Consider now that each $\partial A_i$ is trivial. If one $\partial A_i$ is the meridian of $\partial_{L_i}M(p, 2, r)$, then by capping off $\partial A_i$ with a meridional disk of $N(L_i)$ we obtain a boundary reducible disk for $M_j$, which is not possible since $M_j$ is boundary irreducible as the linking number between $T(\frac{p}{2}, 1)$ and $J$ is greater than zero.
Consider that $\partial A_i$ has the form $T(a, \pm 1)$  and $\partial A_j$ is one of $T(b, \pm 1)$ or $T(\pm 1, b)$. The linking number of $\partial A_i$ and $L_i$ is $\pm 1$, but the linking number of $\partial A_j$ and $L_i$ is $\pm \frac{p}{2}$ or $b\frac{p}{2}$, a contradiction. So, both $\partial A_1$ and $\partial A_2$ are of the form $T(\pm 1, c)$.
Thus, $\partial A_1$ is isotopic to $L_1$ and $\partial A_2$ is isotopic to $L_2$, and so $L_1$ and $L_2$ are isotopic through $A$ in $M(p, 2, r)$, which implies that $r_1 = r_2$ contradicting Lemma~\ref{LinkingNumber}.

Now consider that $\partial A_1 \subset \partial_{L_i}M(p, 2, r)$ and $\partial A_2 \subset \partial_{J}M(p, 2, r)$.
First assume that $\partial A_1$ is the non-trivial torus knot $T(a, b)$. 
Thus, $\partial A_2$ is equal to $T(a, b)$, $T(-a, b)$, $T(a, -b)$, $T(-a, -b)$, $T(b, a)$, $T(-b, a)$, $T(b, -a)$, or $T(-b, -a)$ \cite{Rolfsen}.
The linking number of $\partial A_1$ and $L_i$ is $b$. 
The linking number of $\partial A_2$ and $L_i$ is $ar_i$, $-ar_i$, $br_i$, or $-br_i$. 
It is not possible that $b = ar_i$ or $-ar_i$ since $a, b$ are coprime.
If $b = br_i$ or $-br_i$, then $r_i = 1$ or $-1$. So, $r_i = 1$ and $\partial A_2$ is the torus knot $T(b, \pm a)$.
The linking number of $\partial A_2$ and $J$ is $\pm a$, and the linking number of  $\partial A_1$ and $J$ is $a$ as $r_i = 1$. So, $\partial A_2$ is the torus knot $T(b, a)$.
By moving the $p-3$ last crossings of $T(p, 2)$ anticlockwise around the braid closure, we obtain the braid
$(\sigma_1\sigma_2)^2(\sigma_1)^{p-3}$ and $J$ becomes the braid axis of this braid. 
After switching the meridian to the longitude and the longitude to the meridian of the trivial torus $\partial N(J)$, the link $(\sigma_1\sigma_2)^2(\sigma_1)^{p-3}$ becomes inside a trivial solid torus bounded by $\partial N(J)$ and $\partial A_2$ becomes the torus knot $T(a, b)$. 
The component $L_i$ has wrapping number equal to one inside this solid torus as $r_i = 1$. 
Since $L_j$ doesn't puncture $A$, it must live in the knotted torus $(\partial N(L_i) - N(T(a, b)))\cup A \cup A \cup(\partial N(J) - N(T(a, b)))$ with core the torus knot $T(a, b)$. If $L_j$ has zero wrapping number in this knotted solid torus, then $L_1$ and $L_2$ have zero linking number, a contradiction. Thus, $L_j$ has non-zero wrapping number in this torus, which implies that $L_j$ has an essential torus, which is impossible since $L_j$ is the unknot \cite[page 15]{hatcher2007notes}.

Thus, $\partial A_1$ and $\partial A_2$ are trivial. By one previous argument, both $\partial A_1$ and $\partial A_2$ can't be meridians. Consider that $\partial A_1$, $\partial A_2$ are of the forms $T(\pm 1, a)$, $T(\pm 1, b)$, respectively. Then,
$\partial A_1$, $\partial A_2$ is isotopic to $L_i$, $J$, respectively, and then $L_i$ are $J$ are isotopic to each other. This implies that $r_j = \frac{p}{2}$, a contradiction.
Consider now that $\partial A_1$ has the form $T(\pm 1, a)$ and $\partial A_2$ has the form $T(b, \pm 1)$. The linking number of $\partial A_1$, $\partial A_2$ and $J$ is $\pm r_i, \pm 1$, respectively. Thus, $r_i = 1$, and so $r_j = 2$. 
The linking number of $\partial A_1$, $\partial A_2$ and $L_j$ is $\pm\frac{p}{2}, 2b$, respectively. So, $\vert b\vert >1$ since $p>4$. As $\partial A_1$ and $L_i$ are isotopic and $\partial A_1$ is isotopic to $\partial A_2$, we can isotopy $L_i$ to have the form $T(b, \pm 1)$ in $\partial N(J)$. Then, if we push $L_i = T(b, \pm 1)$ to the inside of $N(J)$, $\partial N(J)$ becomes 
an unknotted essential torus for $T(p, 2)$ as $r_j = 2$ and $\vert b\vert >1$, which is a contradiction with Lemma ~\ref{atoroidaltoruslink}.
If $\partial A_1$ has the form $T(a, \pm 1)$ and $\partial A_2$ has the form $T(\pm 1, b)$, then the linking number of $\partial A_1$ and $L_j$ is $a\frac{p}{2}$ and the linking number of $\partial A_2$ and $L_j$ is $\pm r_j$, which is not possible by the hypotheses.
Finally, if $\partial A_1$ has the form $T(a, \pm 1)$ and $\partial A_2$ has the form $T(b, \pm 1)$, then the linking number of $\partial A_1$ and $L_i$ is $\pm 1$ and the linking number of $\partial A_2$ and $L_i$ is $br_i$ implying that $r_i = 1$ and  $\vert b\vert = 1$, which we have already ruled out.
\end{proof}

\begin{proposition}\label{anannular2}
Consider $p, r$ positive integers such that $p+2 \geq r >1$. Furthermore, assume that $gcd(p, 2) = 2$ and $r$ is odd, different from $p\pm 1$. Then, $M(p, 2, r)$ is anannular.
\end{proposition}

\begin{proof}Assume for contradiction that $M(p, 2, r)$ has an essential annulus $A$. By Lemma~\ref{annulitwocomponent2}, the boundaries of $A$ are in just one boundary component $C$ of $\partial M(p, 2, r)$. We also refer to $C$ as the corresponding link component. Since each boundary component of $M(p, 2, r)$ is unknotted and the unknot torus has no essential annuli \cite[ page 15]{hatcher2007notes}, $A$ is compressible, boundary compressible, or boundary parallel in the exterior of $C$ in $S^3$. Denote by $M_C$ the exterior of $C$ in $S^3$.

Consider first that $A$ is boundary parallel in $M_C$. Then, $A$ is boundary parallel to an annulus $B$ in $C$ so that $A\cup B$ bounds a solid torus $V$ in $M_C$. As $A$ is essential in $M(p, 2, r)$, one component of $\partial M(p, 2, r)$ must be contained in $V$.

Assume first that only one component $K$ of $\partial M(p, 2, r)$ is contained in $V$. Then, $K$ must have wrapping number  greater than zero in $V$ otherwise $M(p, 2, r)$ would be reducible, which is not possible by Lemma~\ref{irreducible}. 
Hence, $V$ is trivial as the unknot doesn't have any essential tori \cite[page 15]{hatcher2007notes}. Furthermore, $B$ can't be a meridional annulus of $C$ otherwise $T(p, 2)$, $M_1$, or $M_2$ would be reducible, but it is not possible by a linking number argument.

Suppose first that $C = L_i$ for $i = 1$ or $2$. 
Denote by $\omega$ the winding number of $K$ in $V$. 
If $K = J$, then the linking number $r_j$ of $J$ and $L_j$ would be either $\omega\frac{p}{2}$ if the core of $B$ has the form $T(\pm 1, a)$ or $a \omega\frac{p}{2}$ if the core of $B$ has the form $T(a, \pm 1)$ because $L_i$ and $L_j$ have linking number equal to $p/2$. However, it contradicts the hypotheses.
Now suppose that $K = L_j$. Then, the linking number  $r_j$ of $L_j$ and $J$ would be either $r_i\omega$ if the core of $B$ has the form $T(\pm 1, a)$ or $a r_i\omega$ if the core of $B$ has the form $T(a, \pm 1)$, which is a contradiction since $r_j$ is equal to $r_i+1$ or $r_i-1$ by Lemma~\ref{LinkingNumber}.

Consider now that $C = J$ and $K = L_i$. Consider first that $\partial V$ is boundary parallel to $\partial N(L_i)$. Then,
the wrapping number of $L_i$ in $V$ is one. If the core of $B$ has the form $T(\pm 1, a)$ in $\partial N(J)$, then $J$ and $L_i$ are isotopic in $M(p, 2, r)$, which implies that they have the same linking number with $L_j$, a contradiction since $r_j\neq \frac{p}{2}$. Assume now that the core of $B$ has the form $T(a, \pm 1)$. So, $\vert a\vert >1$ as $B$ can't be a meridional annulus. We have that $r_i = 1$ as $L_i$ only wraps one time along the meridian of  $\partial N(J)$. So, $r_j = 2$. Thus, if we push $L_i$ to the inside of $N(J)$, $\partial N(J)$ becomes an essential torus for $T(p, 2)$, but $T(p, 2)$ doesn't have an essential tori by Lemma~\ref{atoroidaltoruslink}.
Suppose now that $\partial V$ is not boundary parallel to $\partial N(L_i)$. If $r_j > 1$, then $\partial V$ is essential in $T(p, 2)$, which is impossible by Lemma~\ref{atoroidaltoruslink}. Assume $r_j = 1$. So, $r_i = 2$.
If the core of $B$ is the torus knot $T(\pm 1, a)$, then the winding number $\omega$ of $L_i$ in $V$ is equal to $\frac{p}{2}$ as the linking number between $L_i$ and $L_j$ is equal to $\frac{p}{2}$. Then, $r_i = 2 =  a\omega$, which is impossible since $a$ is an integer and $p>4$. 
For the last case of this case, consider that the core of $B$ has the form of the trivial torus knot $T(a, \pm 1)$ with $\vert a\vert >1$. Then, $\partial V$ is essential in $T(p, 2)$, which contradicts Lemma~\ref{atoroidaltoruslink}.

Suppose now that the other two components of $\partial M(p, 2, r)$ are contained in $V$. They must have wrapping numbers greater than zero in $V$ otherwise $M_1$, $M_2$, or $T(p, 2)$ would be reducible, but it is not possible by a linking number argument. So, $V$ is unknotted. For the same reason, $B$ can't be a longitudinal annulus of $C$.

If the core of $B$ has the form of the trivial torus knot $T(a, \pm 1)$, then the disk with boundary the longitude of $C$ is a boundary compression disk for $A$,  which is not possible since $A$ is essential in $M(p, 2, r)$.
Thus, the core of $B$ must have the form of the trivial torus knot $T(\pm 1, a)$ with $\vert a \vert>1$.
If $\partial V$ is boundary parallel to these components inside $V$ individually, then these components are isotopic in $M(p, 2, r)$. Hence, they have the same linking number with $C$, a contradiction.
Thus, $\partial V$ is not boundary parallel to one of these components. This implies that $\partial V$ is an essential torus in one of $M_1, M_2$, or $S^3-N(T(p, 2))$, depending on where these components belong to. But, this is a contradiction since none of these manifolds have an essential torus by Corollary~\ref{atoroidaltoruslink2} and Lemma~\ref{atoroidaltoruslink}.

Consider then that $A$ is compressible in $M_C$. Then, there is a compression disk $D$ for $A$ in $M_C$. The surgery of $A$ along $D$ yields two disks, $D_1$ and $D_2$, such that $\partial A = \partial D_1 \cup \partial D_2$.
If one $\partial D_i$ bounds a disk on $\partial N(C)$, then by considering a disk with boundary in $A$ close to $\partial D_i$, we can see that $A$ is also compressible in $M(p, 2, r)$, which is not possible. Now suppose that none $\partial D_i$ bounds a disk on $\partial N(C)$. Then, both $\partial D_1$ and $\partial D_2$ are either isotopic to the longitude or meridian of $\partial N(C)$ implying that $A$ is also boundary parallel in $M_C$. However, we have a contradiction with the first case.

Therefore, because $M_C$ is irreducible and $A$ is incompressible and not boundary parallel in $M_C$, the annulus $A$ is essential in $M_C$, which is a contradiction.
\end{proof}

\bibliographystyle{amsplain}

\bibliography{H-T-T-L}

\providecommand{\bysame}{\leavevmode\hbox to3em{\hrulefill}\thinspace}
\providecommand{\MR}{\relax\ifhmode\unskip\space\fi MR }
\providecommand{\MRhref}[2]{%
  \href{http://www.ams.org/mathscinet-getitem?mr=#1}{#2}
}
\providecommand{\href}[2]{#2}
\begin{thebibliography}{10}

\bibitem{newtwis}
Joan Birman and Ilya Kofman, \emph{A new twist on {L}orenz links}, J. Topol.
  \textbf{2} (2009), no.~2, 227--248. \MR{2529294}

\bibitem{Bridge}
Richard~Sean Bowman, Scott Taylor, and Alexander Zupan, \emph{Bridge spectra of
  twisted torus knots}, Int. Math. Res. Not. IMRN (2015), no.~16, 7336--7356.
  \MR{3428964}

\bibitem{simplest}
Patrick~J. Callahan, John~C. Dean, and Jeffrey~R. Weeks, \emph{The simplest
  hyperbolic knots}, J. Knot Theory Ramifications \textbf{8} (1999), no.~3,
  279--297. \MR{1691433}

\bibitem{generalizedtwistedtoruslinks}
Abhijit Champanerkar, David Futer, Ilya Kofman, Walter Neumann, and Jessica~S.
  Purcell, \emph{Volume bounds for generalized twisted torus links}, Math. Res.
  Lett. \textbf{18} (2011), no.~6, 1097--1120. \MR{2915470}

\bibitem{nextsimplest}
Abhijit Champanerkar, Ilya Kofman, and Eric Patterson, \emph{The next simplest
  hyperbolic knots}, J. Knot Theory Ramifications \textbf{13} (2004), no.~7,
  965--987. \MR{2101238}

\bibitem{dePaiva:Unexpected}
Thiago de~Paiva, \emph{Unexpected essential surfaces among exteriors of twisted
  torus knots}, arXiv:2012.09599, to appear in Algebr. Geom. Topol.

\bibitem{de2021hyperbolic}
\bysame, \emph{Hyperbolic knots given by positive braids with at least two full
  twists}, Proc. Amer. Math. Soc. \textbf{150} (2022), no.~12, 5449--5458.
  \MR{4494619}

\bibitem{dePaivaPurcell:SatellitesLorenz}
Thiago de~Paiva and Jessica~S. Purcell, \emph{Satellites and {L}orenz knots},
  arxiv:2103.09500, 2021.

\bibitem{Thesis}
John~Charles Dean, \emph{Hyperbolic knots with small {S}eifert-fibered {D}ehn
  surgeries}, ProQuest LLC, Ann Arbor, MI, 1996, Thesis (Ph.D.)--The University
  of Texas at Austin. \MR{2694392}

\bibitem{GordonToroidal}
Cameron~McA. Gordon and Ying-Qing Wu, \emph{Toroidal and annular {D}ehn
  fillings}, Proc. London Math. Soc. (3) \textbf{78} (1999), no.~3, 662--700.
  \MR{1674841}

\bibitem{GordonWu}
\bysame, \emph{Annular and boundary reducing {D}ehn fillings}, Topology
  \textbf{39} (2000), no.~3, 531--548. \MR{1746907}

\bibitem{GordonAnnular}
\bysame, \emph{Annular {D}ehn fillings}, Comment. Math. Helv. \textbf{75}
  (2000), no.~3, 430--456. \MR{1793797}

\bibitem{Guntel}
Brandy~J. Guntel, \emph{Knots with distinct primitive/primitive and
  primitive/{S}eifert representatives}, J. Knot Theory Ramifications
  \textbf{21} (2012), no.~1, 1250015, 12. \MR{2887904}

\bibitem{hatcher2007notes}
Allen Hatcher, \emph{Notes on basic 3-manifold topology}, 2007.

\bibitem{cable}
Sangyop Lee, \emph{Twisted torus knots {$T(p,q;kq,s)$} are cable knots}, J.
  Knot Theory Ramifications \textbf{21} (2012), no.~1, 1250005, 4. \MR{2887898}

\bibitem{unknotted}
\bysame, \emph{Twisted torus knots that are unknotted}, Int. Math. Res. Not.
  IMRN (2014), no.~18, 4958--4996. \MR{3264672}

\bibitem{LeeTorusknotsobtained}
\bysame, \emph{Torus knots obtained by twisting torus knots}, Algebr. Geom.
  Topol. \textbf{15} (2015), no.~5, 2819--2838. \MR{3426694}

\bibitem{Knottypes}
\bysame, \emph{Knot types of twisted torus knots}, J. Knot Theory Ramifications
  \textbf{26} (2017), no.~12, 1750074, 7. \MR{3718275}

\bibitem{hyperbolicity}
\bysame, \emph{Satellite knots obtained by twisting torus knots: hyperbolicity
  of twisted torus knots}, Int. Math. Res. Not. IMRN (2018), no.~3, 785--815.
  \MR{3801447}

\bibitem{Composite}
\bysame, \emph{Composite knots obtained by twisting torus knots}, Int. Math.
  Res. Not. IMRN (2019), no.~18, 5744--5776. \MR{4012126}

\bibitem{Positively}
\bysame, \emph{Positively twisted torus knots which are torus knots}, J. Knot
  Theory Ramifications \textbf{28} (2019), no.~3, 1950023, 13. \MR{3938086}

\bibitem{torusTwistedtorusknots}
\bysame, \emph{Twisted torus knots t(p, q, p - kq, -1) which are torus knots},
  Journal of Knot Theory and Its Ramifications \textbf{29} (2020), no.~09,
  2050068.

\bibitem{LeeThiago}
Sangyop Lee and Thiago de~Paiva, \emph{Torus knots obtained by negatively
  twisting torus knots}, J. Knot Theory Ramifications \textbf{31} (2022),
  no.~1, Paper No. 2150080, 19. \MR{4411809}

\bibitem{Heegaard}
Yoav Moriah and Eric Sedgwick, \emph{Heegaard splittings of twisted torus
  knots}, Topology Appl. \textbf{156} (2009), no.~5, 885--896. \MR{2498921}

\bibitem{tangle}
Kanji Morimoto, \emph{On tangle decompositions of twisted torus knots}, J. Knot
  Theory Ramifications \textbf{22} (2013), no.~9, 1350049, 12. \MR{3105308}

\bibitem{Anote}
Kanji Morimoto and Yuichi Yamada, \emph{A note on essential tori in the
  exteriors of torus knots with twists}, Kobe J. Math. \textbf{26} (2009),
  no.~1-2, 29--34. \MR{2583175}

\bibitem{Rolfsen}
Dale Rolfsen, \emph{Knots and links}, Publish or Perish, Inc., Berkeley,
  Calif., 1976, Mathematics Lecture Series, No. 7. \MR{0515288}

\bibitem{Thurston}
William~P. Thurston, \emph{Three-dimensional manifolds, {K}leinian groups and
  hyperbolic geometry}, Bull. Amer. Math. Soc. (N.S.) \textbf{6} (1982), no.~3,
  357--381. \MR{648524}

\bibitem{Incompressible}
Chichen~M. Tsau, \emph{Incompressible surfaces in the knot manifolds of torus
  knots}, Topology \textbf{33} (1994), no.~1, 197--201. \MR{1259522}

\bibitem{homology}
Faramarz Vafaee, \emph{On the knot {F}loer homology of twisted torus knots},
  Int. Math. Res. Not. IMRN (2015), no.~15, 6516--6537. \MR{3384486}

\bibitem{Wulaminations}
Ying-Qing Wu, \emph{Sutured manifold hierarchies, essential laminations, and
  {D}ehn surgery}, J. Differential Geom. \textbf{48} (1998), no.~3, 407--437.
  \MR{1638025}

\end{thebibliography}

\end{document}